\documentclass[11pt]{article}
\usepackage{mathtools}
\usepackage{amssymb}
\usepackage{amsthm}
\usepackage{bigints}
\usepackage{mathdots}
\usepackage[]{graphicx}
\usepackage{wrapfig}
\usepackage{float}
\usepackage{verbatim}
\usepackage[margin=1in]{geometry}
\usepackage[utf8]{inputenc}
\usepackage[T1]{fontenc}
\newtheorem{theorem}{Theorem}

\newtheorem{remark}[theorem]{Remark}
\newtheorem{proposition}[theorem]{Proposition}
\newtheorem{lemma}[theorem]{Lemma}
\title{The Geometry of Axisymmetric Ideal Fluid Flows with Swirl}
\author{Pearce Washabaugh and Stephen C. Preston}
\newcommand{\tab}{\hspace*{2em}}
\DeclareMathOperator{\curl}{curl}

\newcommand{\Diffmu}{\mathcal{D}_{\mu}}
\newcommand{\DiffmuE}{\mathcal{D}_{\mu,E}}
\newcommand{\DiffmuF}{\mathcal{D}_{\mu,F}}
\newcommand{\id}{\text{id}}
\newcommand{\Reeb}{E}
\newcommand{\llangle}{\langle\!\langle}
\newcommand{\rrangle}{\rangle\!\rangle}

\date{8/27/2014}
\begin{document}

\maketitle

\begin{abstract}

 The sectional curvature of the volume preserving diffeomorphism group of a Riemannian manifold $M$ can give information about the stability of  inviscid, incompressible fluid flows on $M$. We demonstrate that the submanifold of the volumorphism group of the solid flat torus generated by axisymmetric fluid flows with swirl, denoted by $\DiffmuE(M)$, has positive sectional curvature in every section containing the field $X = u(r)\partial_\theta$ iff $\partial_r(ru^2)>0$.
This is in sharp contrast to the situation on $\Diffmu(M)$, where only Killing fields $X$ have nonnegative sectional curvature in all sections containing it. We also show that this criterion guarantees the existence of conjugate points on $\DiffmuE(M)$ along the geodesic defined by $X$.

\end{abstract}

\section*{Introduction}
\tab Let $(M,g)$ be a Riemannian manifold of dimension at least two with Riemannian volume form $\mu$. The configuration space for inviscid, incompressible fluid flows on $M$ is the collection of volume-preserving diffeomorphisms (volumorphisms) of $M$, denoted by $\mathcal{D}_\mu(M)$. Arnold \cite{Arnold66} demonstrated in 1966 that $\mathcal{D}_\mu(M)$ can be thought of as an infinite dimensional Riemannian manifold. He also showed that flows obeying the Euler equations for inviscid, incompressible fluid flow can be realized as geodesics on $\mathcal{D}_\mu(M)$. This was proved rigorously in the context of Sobolev manifolds by Ebin and Marsden \cite{Ebin70}. Using this framework, questions of fluid mechanics can be re-phrased in terms of the Riemannian geometry of $\Diffmu(M)$. A good account of this is given in \cite{Arnold98} or more recently in \cite{KLMP2}. Of particular interest is the sectional curvature of $\Diffmu(M)$. As in finite dimensional geometry, given two geodesics with varying initial velocities in a region of strictly positive (resp. negative) sectional curvature, the two geodesics will converge (resp. diverge) via the Rauch Comparison theorem. In terms of fluid mechanics, this corresponds to the Lagrangian stability (resp. instability) of the associated fluid flows.\\
\tab Arnold showed that the sectional curvature $K(X,Y)$ of the plane in $T_{\id}\Diffmu(M)$ spanned by $X$ and $Y$  is often negative but occasionally positive. Rouchon~\cite{Rouchon91} sharpened this to show that if $M\subset \mathbb{R}^3$, then $K(X,Y)\ge 0$ for every $Y\in T_{\id}\Diffmu(M)$ if and only if $X$ is a Killing field (i.e., one for which the flow generates a family of isometries). This result was generalized by Misio{\l}ek \cite{Misiolek92} and the second author \cite{Preston02} for any manifold with $\dim{M}\ge 2$. This gives the impression that, in general, $D_\mu(M)$ will mostly be negatively curved. The question of when one can expect a divergence free vector field to give nonpositive sectional curvature remains open. However, the second author \cite{Preston05} provided criteria for divergence free vector fields of the form $X = u(r)\partial_\theta$ on the area-preserving diffeomorphism groups of a rotationally-symmetric surface for which the sectional curvature $K(X,Y)$ is nonpositive for all $Y$.\\
\tab Our goal in this paper is to extend the curvature computation to $\DiffmuE(M)$, the group of volumorphisms commuting with the flow of a Killing field $\Reeb$. In particular, we consider the solid flat torus, $M= D^2\times S^1$, where $D^2$ is the unit disk in $\mathbb{R}^2$ and $S^1$ is the unit circle, as a subset of $\mathbb{R}^3$ with the planes $z=0$ and $z=2\pi$ identified, where $E = \partial_\theta$ is the field corresponding to rotation in the disc. One may imagine fluid flows on this manifold as axisymmetric ideal flows with swirl on the solid infinite cylinder which are $2\pi$-periodic in the $z$-direction. We consider steady fluid velocity fields of the form $X = u(r)\partial_\theta$. The submanifold  $\DiffmuE(M)$ is a totally geodesic submanifold of $\Diffmu(M)$, corresponding to the fact that an ideal fluid which is initially independent of $\theta$ will always remain so. Hence we compute sectional curvatures $K(X,Y)$ where $Y\in T_{\id}\DiffmuE(M)$ is divergence-free \emph{and} axisymmetric, i.e., $[\Reeb,Y]=0$.

\tab In \cite{Preston05} we effectively showed that when $X$ was considered as an element of $\DiffmuF(M)$ where $F = \frac{\partial}{\partial z}$ (corresponding to considering $X$ as a two-dimensional flow rather than a three-dimensional flow), the sectional curvature satisfied $K(X,Y)\le 0$ for every $Y\in T_{\id}\DiffmuF(M)$ regardless of $u(r)$. By contrast we show here that if $u$ satisfies the condition
\begin{equation}\label{positivecurvature}
\frac{d}{dr}\big( ru(r)^2\big) >  0,
\end{equation}
then $K(X,Y)> 0$ for every $Y\in T_{\id}\DiffmuE(M)$. We will also show that $\frac{d}{dr}\big( ru(r)^2\big) \geq  0$ implies that $K(X,Y)\geq 0$. This does not contradict the result of Rouchon, since the proof of that result relies on being able to construct a divergence-free velocity field with small support which points in a given direction and is orthogonal to another direction, and there are not enough divergence-free vector fields in the axisymmetric case to accomplish this here.

\tab The fact that the curvature is strictly positive in every section containing $X$ makes it natural to ask whether there are conjugate points along every such corresponding geodesic. Unfortunately the Rauch comparison theorem cannot be used here, since $\inf_{Y\in T_{\id}\DiffmuE(M)} K(X,Y) = 0$ even if \eqref{positivecurvature} holds. Nonetheless we can show that as long as
\begin{equation}\label{conjugatepoints}
ru(r)u'(r) + 2u(r)^2 > 0,
\end{equation}
the geodesic formed by $X = u(r)\partial_{\theta}$ has infinitely many monoconjugate points. It is easy to see that condition \eqref{positivecurvature} implies \eqref{conjugatepoints}. We do this by solving the Jacobi equation explicitly. As in \cite{EMP}, where the case $u(r) = 1$ was considered, we can prove that these monoconjugate points have an epiconjugate point as a limit point, so that the differential of the exponential map is not Fredholm.

The second author gratefully acknowledges support from NSF grants DMS-1157293 and DMS-1105660.


\section*{The Formula for Curvature}

We first compute the curvature of $\DiffmuE(M)$ by expanding in a Fourier series in $z$. Notice first of all that any vector field $Y$ which
is tangent to $\DiffmuE(M)$ at the identity must be divergence-free and must commute with $E=\frac{\partial}{\partial \theta}$. Therefore we
can write in the form
\begin{equation}\label{Ytangent}
Y(r,z) = -\frac{g_z(r,z)}{r} \, \partial_r + \frac{g_r(r,z)}{r} \, \partial_z + f(r,z) \, \partial_{\theta},
\end{equation}
where $f(0,z)=g(0,z) = 0$ and $g(1,z)$ is constant in $z$ (in order to be well-defined on the axis of symmetry and to have
$Y$ tangent to the boundary $r=1$). We think of the term $-\frac{g_z}{r} \partial_r + \frac{g_r}{r} \partial_z$ as an analogue of the skew-gradient
in two dimensions. We may express $Y$ in a Fourier series in $z$ as
$ Y(r,z) = \sum_{n\in\mathbb{Z}} Y_n(r,z)$
where
\begin{equation}\label{Ytangentfourier}
Y_n(r,z) = e^{inz} \left[ -\frac{in}{r} g_n(r) \, \partial_r + \frac{g_n'(r)}{r} \, \partial_z + f_n(r) \, \partial_\theta\right].
\end{equation}

On any Riemannian manifold $(M,g)$ with volume form $\mu$, a formula for the curvature tensor on $\mathcal{D}_\mu(M)$ is given by

\begin{equation}\label{generalcurvature}
R(Y,X)X = P\left(\nabla_YP(\nabla_XX)-\nabla_XP(\nabla_YX)+\nabla_{[X,Y]}X\right),
\end{equation}

where $P(X)$ is the projection onto the divergence-free part of $X$. Concretely, $P(X)$ is obtained by solving the Neumann boundary value problem

$$\begin{cases}

   \Delta q = \mbox{div }X & \mbox{ in }M\\

   \left<\nabla q, \vec{n}\right> = \left<X,\vec{n}\right> & \mbox{ on }\partial M

  \end{cases}
$$

for $q$ and then setting $P(X) = X - \nabla q$. The non-normalized sectional curvature is then given by

\begin{equation}\label{nonnormalizeddef}
\overline{K}(X,Y) = \llangle R(Y,X)X,\overline{Y}\rrangle = \int_M \left<R(Y,X)X,\overline{Y}\right> \mu.
\end{equation}

See \cite{Misiolek92} for the derivation of the formula we use here. We first compute $R(Y_n,X)X$.


\begin{proposition}\label{firstcurvatureformula}
Let $M = D^2\times S^1$. Suppose that $X\in T_{\id}\DiffmuE(M)$ is defined by $X = u(r)\partial_\theta$,
and let $Y_n$ be of the form \eqref{Ytangentfourier}. Then the curvature tensor $R(Y_n,X)X$ is given by
\begin{equation}\label{curvaturecomponent}
R(Y_n,X)X = P\left( -inug_n(2u'+\frac{u}{r})e^{inz}\partial_r+\frac{(q_n'+rf_nu)ue^{inz}}{r}\partial_\theta  \right),
\end{equation}
where $q_n$ is the solution of the ODE
\begin{equation}\label{qnODE}
\begin{cases}
\frac{1}{r}\,\frac{d}{dr}\left( r\,\frac{dq_n}{dr}\right) - n^2 q_n(r)  = -\frac{1}{r} \, \frac{d}{dr} \big( r^2 f_n(r) u(r)\big) \text{ for $0<r<1$} \\
q_n'(1) =  -f_n(1)u(1) \\
|q_n(0)|<\infty
\end{cases}
\end{equation}
\end{proposition}

\begin{proof}
We compute using formula \eqref{generalcurvature}. First note that $\nabla_X X  = -ru^2\partial_r$, which is the gradient of a function.
Thus $P(\nabla_X X ) = 0$. Next for $n=0$ note that
%
%
%
%
%
$$Y_0 = \tfrac{1}{r} \, g_0'(r) \partial_z + f_0(r)\partial_\theta$$
and
$\nabla_{Y_0}X = -rf_0(r)u(r)\partial_r$.

This is also the gradient of a function, and thus

$$P(\nabla_{Y_0}X) = 0.$$

Now for $n\neq 0$,

$$\nabla_{Y_n}X = -rf_nu e^{inz}\partial_r - \frac{in}{r}g_n(u'+\frac{u}{r})e^{inz}\partial_\theta.$$

The solution $q_n(r)e^{inz}$ of
\begin{equation}\label{qnPDE}
\begin{cases}
\Delta (q_n(r) e^{inz}) = \mbox{div}(\nabla_{Y_n}X) &\mbox{in }M, \\
\left.\left<\nabla (q_ne^{inz}),\vec{n}\right> \right|_{\partial M}= \left.\left<\nabla_{Y_n}X,\vec{n}\right>\right|_{\partial M}& \mbox{on }\partial M,
\end{cases}
\end{equation}
clearly must satisfy \eqref{qnODE}. With this solution in hand we will get

$$\nabla_X(P(\nabla_{Y_n}X)) = inug_n\left(u'+\frac{u}{r}\right)e^{inz} \partial_r- \frac{(q_n'+rf_nu)ue^{inz}}{r}\,\partial_\theta. $$

We also easily compute

$$\nabla_{[X,Y]}X = -in g_n(r)u(r)u'(r)e^{inz} \, \partial_r.$$

So, $R$ will be given by \eqref{curvaturecomponent}.
\end{proof}

To get a more explicit and useful formula for curvature, we proceed to solve the ODE \eqref{qnODE}.
%
%
%
%
%
%
%
\begin{lemma}\label{qnODEsoln}
If $q_n$ satisfies \eqref{qnODE}, then the solution is given by
\begin{equation}\label{qnformula}
q_n(r) = -\zeta_n(r) H_n(r) + \xi_n(r) J_n(r),
\end{equation}
where
\begin{equation}
H_n(r) = \int_0^r s^2 f_n(s) u(s) \xi_n'(s) \, ds \qquad\text{and}\qquad
J_n(s) = -\int_r^1 s^2 f_n(s) u(s) \zeta_n'(s) \, ds,\label{HnJndef}
\end{equation}
and
$$ \xi_n(r) = I_0(nr) \qquad \text{and}\qquad \zeta_n(r) = \tfrac{K_1(n)}{I_1(n)} I_0(nr) + K_0(nr),$$
with $I_0$ and $K_0$ denoting the modified Bessel functions of the first and second kinds.
\end{lemma}

\begin{proof}
Since $I_0$ and $K_0$ solve the homogeneous version of \eqref{qnODE}, this is essentially just the
variation of parameters formula together with an integration by parts. 
We simply verify the solution: taking the derivative of $q_n(r)$, we obtain
\begin{equation}\label{qnprime}
\begin{split}
q_n'(r) &= -\zeta_n'(r) H_n(r) + \xi_n'(r) J_n(r) + r^2 f_n(r)u(r) \big( \xi_n(r)\zeta_n'(r) - \zeta_n(r) \xi_n'(r)\big) \\
&= -\zeta_n'(r) H_n(r) + \xi_n'(r) J_n(r) - r f_n(r) u(r),
\end{split}
\end{equation}
and since $\zeta_n'(1)=J_n(1)=0$ we get the correct boundary condition. Furthermore we get
$$ q_n''(r) = -\zeta_n''(r) H_n(r) + \xi_n''(r) J_n(r) - \frac{d}{dr} \big( r f_n(r) u(r)\big),$$
and with these formulas we easily check that $q_n$ satisfies \eqref{qnODE}.
\end{proof}


Plugging in the formula for $q_n$ from Lemma \ref{qnODEsoln} to the formula from Proposition \ref{firstcurvatureformula},
we obtain a very simple result.

\begin{theorem}\label{maintheorem1}
On $M=D^2\times S^1$ with $X = u(r) \, \partial_{\theta}$ and $Y$ expressed as in \eqref{Ytangent}, the non-normalized
sectional curvature is given by
$ \overline{K}(X,Y) = \sum_{n\in \mathbb{Z}} \overline{K}(X,Y_n)$, where
$Y_n$ is expressed as in \eqref{Ytangentfourier} and
\begin{equation}\label{finalcurvatureformula}
\overline{K}(X,Y_n) = 4\pi^2 \int_0^1 \frac{1}{r} \left( n^2 \, \lvert g_n(r)\rvert^2 \, \frac{d}{dr} \big( ru(r)^2\big) + \frac{\lvert H_n(r)\rvert^2}{I_1(nr)^2}\right) \, dr.
\end{equation}
Hence the curvature is positive for all $Y$ if and only if $\frac{d}{dr} \big(r u(r)^2\big) > 0$.
\end{theorem}

\begin{proof}
Using formula \eqref{qnprime} in \eqref{curvaturecomponent}, we obtain
$$R(Y_n,X)X= P\left[(-inug_n(2u'+\frac{u}{r})e^{inz})\partial_r+\frac{u(\xi_n'J_n-\zeta_n'H_n)e^{inz}}{r}\partial_\theta\right],$$
which can clearly be expressed as $e^{inz}$ times a function of $r$ only.
Orthogonality of the functions $e^{imz}$ and $e^{inz}$ over $S^1$ when $m\ne n$ implies that
$$ \overline{K}(X,Y) = \sum_{m,n\in\mathbb{Z}} \llangle Y_m, R(Y_n, X)X\rrangle = \sum_{n\in\mathbb{Z}} \llangle \overline{Y_n}, R(Y_n, X)X\rrangle = \sum_{n\in\mathbb{Z}} \overline{K}(X,Y_n).$$

The latter is now relatively easy to compute. We have
\begin{equation}\label{curvaturefourier}
\overline{K}(X,Y_n) = 4\pi^2 \int_0^1 \frac{n^2}{r^2} \lvert g_n(r)\rvert^2\, \eta(r) \, dr
+ 4\pi^2 \int_0^1 r^2\overline{f_n(r)} \Big( u(r)\big( \xi_n'(r)J_n(r) - \zeta'(r)H_n(r)\big)\Big) \,dr,
\end{equation}
where $\eta(r) = \frac{d}{dr} \big( ru(r)^2\big)$.
By the definitions \eqref{HnJndef} of $H_n$ and $J_n$, we see that the second term in
\eqref{curvaturefourier} is
$$ 4\pi^2 \int_0^1 \big(\overline{H_n'(r)} J_n(r) - \overline{J_n'(r)} H_n(r)\big)\,dr.$$

%

%
%
%
%
%
%
%
%
%
%
%
From here we adapt the corresponding computation in \cite{Preston06}.
%
Integrating by parts and using the fact that $J_n(r)\overline{H_n}(r)\to 0$ as $r\to 0$ or $r\to 1$, we get
%

%
%
%
%
%
$$
\int_0^1\overline{H_n'(r)}J_n(r)-\overline{J_n'(r)}H_n(r)\,dr  = -2\mbox{Re}\int_0^1\overline{J_n'(r)}H_n(r)\,dr = \int_0^1\frac{J_n'(r)}{H_n'(r)}\frac{d}{dr}\left(\left|H_n(r)\right|^2\right)\,dr,$$
and another integration by parts (where again the boundary terms vanish) gives
$$ \int_0^1\overline{H_n'(r)}J_n(r)-\overline{J_n'(r)}H_n(r)\,dr = -\int_0^1\frac{d}{dr}\left(\frac{K_1(nr)}{I_1(nr)}\right) \lvert H_n(r)\rvert^2 \, dr.
$$
Finally the Bessel function identity $\frac{d}{dr} \left( \frac{K_1(r)}{I_1(r)}\right) = \frac{1}{r I_1(r)^2}$ implies \eqref{finalcurvatureformula}.
\end{proof}
%
%
%
%
%
%


\begin{remark}\label{lowerbound}
The normalized sectional curvature is given by
$ K(X,Y) = \frac{\overline{K}}{\llangle X,X\rrangle \llangle Y,Y\rrangle - \llangle X,Y\rrangle^2}$. Suppose that $f=0$ and that only one $g_n$ is nonzero in
\eqref{Ytangentfourier}; then we have $\llangle X,Y\rrangle = 0$ and the sectional curvature takes the form
$$ K(X,Y) = \frac{n^2 \int_0^1 \frac{1}{r} \, \lvert g_n(r)\rvert^2 \, \frac{d}{dr} \big( ru(r)^2\big) \, dr}{\left(\int_0^1 r^3u(r)^2\,dr\right) \left( \int_0^1 \big(\frac{n^2}{r} \lvert g_n(r)\rvert^2 + \lvert g_n'(r)\rvert^2 \big) \, dr\right)}.$$
We can make this arbitrarily small by choosing a highly oscillatory $g_n$. Hence although the curvature is strictly positive if $\frac{d}{dr} \big(ru(r)^2\big) > 0$, it cannot be
bounded below by any positive constant.
\end{remark}

\section*{Solution of the Jacobi equation}

It is natural to ask whether the positive curvature guaranteed by the theorem above ensures the existence of conjugate points
along the corresponding geodesic. This is not automatic since although the sectional curvature is positive in all sections
containing the geodesic's tangent vector, it is not bounded below by any positive constant because of Remark \ref{lowerbound}; hence the Rauch comparison theorem
cannot be applied directly. In this section we answer this question affirmatively by solving the Jacobi equation more or
less explicitly along such a geodesic, and show that in fact conjugate points occur rather frequently.

\begin{theorem}\label{conjugatepointtheorem}
Let $\eta(t)$ be a geodesic on $\DiffmuE(D^2\times S^1)$
with initial condition $\eta(0)=\mbox{id}$ and $\dot{\eta}(0) = X = u(r)\partial_\theta$.
Let $\omega(r) = 2u(r) + ru'(r)$ denote the vorticity function of $X$, and assume that $u(r)\omega(r)>0$ for all $r\in [0,1]$.
Then $\eta(t)$ is a monoconjugate point to $\eta(0)$ for every time $t = 2 \pi \lambda/n$, where $n\in\mathbb{N}$ is arbitrary and $\lambda$ is any eigenvalue of
the Bessel-type Sturm-Liouville problem
$$ \frac{1}{r} \, \frac{d}{dr}\big( r\psi'(r)\big) - \Big( n^2 + \frac{1}{r^2}\Big) \psi(r) = -2 \lambda^2 u(r) \omega(r) \psi(r), \qquad \psi(1)=0, \quad\psi(0) \text{ finite}.$$
\end{theorem}

\begin{proof}
Along a geodesic $\eta(t)$ with (steady) Eulerian velocity field $X$, the Jacobi equation for a Jacobi field $J(t) = Y(t)\circ\eta(t)$
may be written~\cite{Preston02} as the system
\begin{align}
\frac{\partial Y}{\partial t} + [X,Y(t)] &= Z(t) \label{linearizedflow}\\
\frac{\partial Z}{\partial t} + P(\nabla_XZ(t) + \nabla_{Z(t)}X) &= 0, \label{linearizedeuler}
\end{align}
where $P$ is the orthogonal projection onto divergence-free vector fields.
The first equation is the linearized flow equation, while the second is the linearized Euler equation used in stability analysis.

Write
$$ Z(t,r,z) = -\frac{1}{r} \, \frac{\partial h}{\partial z}(t,r,z) \, \partial_r + \frac{1}{r} \, \frac{\partial h}{\partial r}(t,r,z)\,\partial_z + j(t,r,z) \, \partial_\theta,$$
where $h=0$ on the axis $r=0$ and $h$ is constant on the boundary $r=1$.
Then it is easy to compute that \eqref{linearizedeuler} becomes the system
\begin{align}
\frac{\partial j}{\partial t}(t,r,z) &= \frac{\omega(r)}{r^2} \, \frac{\partial h}{\partial z}(t,r,z), \label{jeq} \\
-\frac{1}{r} \, \frac{\partial^2 h}{\partial t\partial z}(t,r,z) \, \partial_r + \frac{1}{r} \, \frac{\partial^2 h}{\partial t\partial r}(t,r,z) \, \partial_z &= 2 P\left( r u(r) j(t,r,z) \, \partial_r\right), \label{heq}
\end{align}
where $\omega(r) = 2u(r) + r u'(r)$ is the vorticity defined by $\curl X = \omega(r) \, \partial_z$.
Applying the curl to both sides of equation \eqref{heq} to eliminate the projection operator, we obtain
\begin{equation}\label{hlaplacian}
\frac{\partial}{\partial t} \left[ \frac{\partial}{\partial r}\left( \frac{1}{r} \, \frac{\partial h}{\partial r}\right) + \frac{1}{r} \, \frac{\partial^2h}{\partial z^2} \right]
= -2r u(r) \, \frac{\partial j}{\partial z}.
\end{equation}

Differentiating \eqref{hlaplacian} in time and substituting \eqref{jeq} we obtain the single equation
\begin{equation}\label{hsingle}
\frac{\partial^2}{\partial t^2}\left[ \frac{\partial}{\partial r}\left( \frac{1}{r} \, \frac{\partial h}{\partial r}\right) + \frac{1}{r} \, \frac{\partial^2 h}{\partial z^2}\right]
= -\frac{2u(r) \omega(r)}{r}\, \frac{\partial^2 h}{\partial z^2}.
\end{equation}
Expand $h$ in a Fourier series in $z$ to get
$$ h(t,r,z) = \sum_{n\in \mathbb{Z}} h_n(t,r) e^{inz}.$$
Then for each $n$ we can solve the eigenvalue problem
$$ \frac{d}{dr} \left( \frac{1}{r} \, \phi'(r)\right) - \frac{n^2}{r} \, \phi(r) = \frac{2C u(r) \omega(r)}{r} \, \phi(r);$$
to make this look more familiar we set $\phi(r) = r \psi(r)$ and obtain
$$ \frac{1}{r} \frac{d}{dr} \big( r \psi'(r)\big) - \Big( n^2 + \frac{1}{r^2}\Big) \psi(r) = 2C u(r) \omega(r) \psi(r),$$
which is a singular Sturm-Liouville problem analogous to the Bessel equation.
We obtain a sequence of eigenfunctions $\phi_{mn}(r)$ for $m\in \mathbb{N}$, with eigenvalues $C_{mn}$. We see that
$$ 2 C \int_0^1 \frac{u(r) \omega(r)}{r}\, \phi(r)^2 \, dr = -\int_0^1 \frac{1}{r} \, \phi'(r)^2 \, dr - \int_0^1 \frac{n^2}{r} \, \phi(r)^2 \, dr,$$
so that if $\omega(r)u(r)>0$, then  $C$ must be strictly negative; we write $C = -\lambda_{mn}^2$ for the eigenfunction $\phi_{mn}(r)$. Expanding $h_n(t,r)$ in a basis of such
eigenfunctions as
$$ h(t,r,z) = \sum_{n\in\mathbb{Z}} \sum_{m=1}^{\infty} h_{mn}(t) \phi_{mn}(r) e^{inz},$$
equation \eqref{hsingle} becomes
$$ -\lambda_{mn}^2 h_{mn}''(t) = n^2 h_{mn}(t),$$
which obviously has solutions
$$ h_{mn}(t) = a_{mn} \cos{\left( \frac{nt}{\lambda_{mn}}\right)} + b_{mn} \sin{\left( \frac{n t}{\lambda_{mn}}\right)}$$
for some coefficients $a_{mn}$ and $b_{mn}$.

Suppose $a_{m,n}=a_{m,-n}=\tfrac{1}{2}$ for some $(m,n)$ with $n\ne 0$, and that all other $a$ are zero and that every $b$ is zero,
so that $h(t,r,z) = \cos{\left( \frac{nt}{\lambda_{mn}}\right)} \phi_{mn}(r) \cos{nz}$.
Then by equation \eqref{hlaplacian} we compute that
$$ j(t,r,z) = -\frac{\lambda_{mn} \omega(r)}{r^2} \, \phi_{mn}(r) \sin{nz} \sin{\left( \frac{nt}{\lambda_{mn}}\right)}.$$

To find the Jacobi fields, write $Y$ in equation \eqref{linearizedflow} as
$$ Y(t,r,z) = -\frac{1}{r} \, \frac{\partial g}{\partial z}(t,r,z) \, \partial_r + \frac{1}{r} \, \frac{\partial g}{\partial r}(t,r,z) \, \partial_z + f(t,r,z) \, \partial_\theta.$$
We easily compute that $X = u(r) \, \frac{\partial}{\partial \theta}$ gives $[X,Y] = \frac{1}{r} \frac{\partial g}{\partial z} u'(r) \, \frac{\partial}{\partial\theta}$, and thus equation \eqref{linearizedflow} becomes in components
\begin{align*}
\frac{\partial g}{\partial t}(t,r,z) &= h(t,r,z)  \\
\frac{\partial f}{\partial t}(t,r,z) + \frac{u'(r)}{r} \, \frac{\partial g}{\partial z}(t,r,z) &= j(t,r,z).
\end{align*}
With $g(0,r,z)=f(0,r,z)=0$, we find that
\begin{align*}
g(t,r,z) &= \frac{\lambda_{mn}}{n} \, \cos{nz} \sin{\left( \frac{nt}{\lambda_{mn}}\right)} \phi_{mn}(r) \\
f(t,r,z) &= \frac{2\lambda_{mn}^2u(r)}{nr^2}\, \sin{nz} \left( \cos{\left( \frac{nt}{\lambda_{mn}}\right)} - 1\right) \phi_{mn}(r).
\end{align*}
Thus both $f$ and $g$ vanish when $t=0$ and when $t = 2 \pi \lambda_{mn}/n$, so $\eta(2\pi\lambda_{mn}/n)$ is monoconjugate to the identity along $\eta$.
\end{proof}

\begin{remark}
Using the Sturm comparison theorem we can estimate the spacing of the eigenvalues $\lambda_{mn}$ and show that for fixed $m$ the sequence $\lambda_{mn}/n$ has a finite limit as $n\to\infty$. Just as in \cite{EMP}, this must be an epiconjugate point. Therefore the differential of the exponential map is not Fredholm along any geodesic of this form. It is
worth noting that the reason the Jacobi equation is explicitly solvable in this case is because there is no ``drift'' term, so the total time derivative agrees with the partial
time derivative, in the same way as in \cite{EMP}.
\end{remark}

It would be very interesting to generalize the curvature computation to fields of the form $X = u(r) \sin{z} \, \partial_{\theta}$,
which is the initial velocity field of the Hou-Luo initial condition \cite{LH} that leads numerically to a blowup solution.
We expect that the formula $\int \overline{H_n'}J_n - \overline{J_n'}H_n$ which appears both here and in \cite{Preston05} is
a typical feature of curvature formulas when computed correctly, although they doubtless become substantially more complicated.

\end{document}